\documentclass[12pt]{article}
\usepackage{amsmath,latexsym}
\usepackage{amsthm}
\usepackage{amssymb}
\usepackage{leftidx}
\usepackage{wasysym}
\usepackage[noblocks]{authblk}
\def\q{\quad}
\def\qq{\qquad}

\def\f{\frac}

\def\b{\binom}

\def\G#1{\Gamma(#1)}

\def\Fa{\Bbb F}

\def\Fq{\widehat{\Bbb F_q}}

\allowdisplaybreaks
\numberwithin{equation}{section}
\theoremstyle{plain}
\newtheorem{thm}{Theorem}[section]
\theoremstyle{definition}
\newtheorem{df}{Definition}[section]
\theoremstyle{Proposition}
\newtheorem{pro}{Proposition}[section]
\theoremstyle{lemma}

\theoremstyle{Corollary}
\newtheorem{cor}{Corollary}[section]

\begin{document}
%\begin{CJK}{GBK}{song}
\title{\bf \Large Some new formulas for Appell series\\ over finite fields
\thanks{This work was supported by the Natural Science Foundation of China (grant No.11571114).}}
\date{}
\author{Long Li\thanks{Corresponding author.\newline \indent
   E-mail addresses: lilong6820@126.com (Long Li), lixin8138@163.com (Xin Li), maorui1111@163.com (Rui Mao).}, Xin Li and Rui Mao\\
   \small{Department of Mathematics, East China Normal University, 500 Dongchuan Road, Shanghai, 200241, PR China}}
%\author[add1]{Aaron\corauthref{cor1}}
%\author{\CJKfamily{fs}{Long Li$^{a,\star}$, Xin Li$^b$ and Rui Mao$^c$ }
%\\[2pt]
%{\small $^{a,b,c}$ Department of Mathematics, East China Normal University, 500 Dongchuan Road,}\\
%{\small Shanghai, 200241, PR China}\\}
%\author {Long Li\\Department of Mathematics, East China Normal University,\\
%Shanghai, 200241, PR China\\Email:\ lilong6820@126.com.\\
%Xin Li\\Department of Mathematics, East China Normal University,\\
%Shanghai, 200241, PR China\\Email:\ lixin8138@163.com. }
\maketitle
\begin{abstract} In 1987 Greene introduced the notion of the finite field analogue of hypergeometric series. In this paper
we give a finite field analogue of Appell series and obtain some
transformation and reduction formulas. We also establish the
generating functions for Appell series over finite fields.
\end{abstract}
\section{Introduction}
 \q  We first introduce some notation. Let $p$ denote a prime and $r$ a positive integer. Set $q=p^r$
 and denote $\Bbb F_q$ as the finite field of $q$ elements. Let $\widehat{\Bbb F_q^*}$
 be the group of multiplicative characters of $\Bbb F_q^*$. We extend the domain of all
 characters $\chi$ of $\Bbb F_q^*$ to $\Bbb F_q$ by defining
$\chi(0)=0$ including trivial character $\varepsilon$, and denote
$\overline\chi$ as the inverse of $\chi$. For a more detailed
introduction to characters, see \cite[Chapter 8]{KM} and \cite{BRE}.
We note that although \cite{BRE} contains proofs of most results
involving characters, care must be taken for cases involving the
trivial character $\varepsilon$, which is defined to be $1$ at $0$
in this reference.

 Following Bailey \cite{B2}, the classical generalized hypergeometric series is defined as
\begin{equation}\label{eq1.1}_{r+1}F_s\left[\begin{array}{cccc}a_0,&a_1,&\ldots,&a_r\\&b_1,&\ldots,&b_s
\end{array};z\right]=\sum_{k=0}^\infty\f{(a_0)_k
(a_1)_k\cdots(a_r)_k}{(b_1)_k\cdots(b_s)_k}\f{z^k}{k!},\end{equation} where $(a)_k=a(a+1)\cdots (a+k-1).$

In 1987, Greene \cite{GJ} developed the theory of hypergeometric
functions over finite fields and he also proved a number of
transformation and summation identities for hypergeometric series
over finite fields that are analogues of those in the classical
case. For characters $A,B,C\in\Fq$ and $x\in\Bbb F_q$, Greene
\cite{GJ}
 defined $$_2\Bbb F_1\left[\begin{array}{rr}A,&B\\&C\end{array};x\right]^G=\varepsilon(x)
\f{BC(-1)}q\sum_uB(u)\overline BC(1-u)\overline A(1-ux)$$ as the
finite field analogue of the integral representation of Gauss
hypergeometric series \cite{ARR, B2}:
$$_2F_1\left[\begin{array}{rr}a,&b\\&c\end{array};x\right]=\f{\G c}{\G b\G{c-b}}
\int_0^1t^b(1-t)^{c-b}(1-tx)^{-a}\f{dt}{t(1-t)}.$$
 And he also introduced $$\b AB^G=\f{B(-1)}qJ(A,\overline B),$$ where $J(\cdot,\cdot)$ denotes the classical Jacobi
 sum. In fact, $\b AB^G$ could be regarded as the finite field analogue of the binomial
coefficient.
 For the sake of simplicity, we define $$\b AB=B(-1)J(A,\overline{B})=q\b AB^G$$ and \begin{equation}\label{eq1.3}
  _2\Bbb F_1\left[\begin{array}{rr}A,&B\\&C\end{array};x\right]=q\ _2\Bbb F_1\left[\begin{array}{rr}A,&B\\&C\end{array};x\right]^G.\end{equation}
Now we restate two theorems of Greene in our notation.
\begin{thm}\label{thm1.1}\rm{\cite[Theorem 3.6]{GJ}} For characters $A,B,C\in \Fq$ and $x\in \Bbb F_q$ we have
$$ _2\Bbb F_1\left[\begin{array}{rr}A,&B\\&C\end{array};x\right]= \f1{q-1}\sum_\chi\b{A\chi}\chi\b{B\chi}{C\chi}\chi(x).$$
\end{thm}
\begin{thm}\label{tm1.2}\rm{\cite[Theorem 4.9]{GJ}} For characters $A,B,C\in \Fq$ we have
$$ _2\Bbb F_1\left[\begin{array}{rr}A,&B\\&C\end{array};1\right]= A(-1)\b{B}{\overline AC}.$$
\end{thm}
Theorem 1.2 is the finite field analogue of Gauss's evaluation
\cite{ARR, B2}:
$$_2F_1\left[\begin{array}{rr}a,&b\\&c\end{array};1\right]=\f{\G c\G{c-a-b}}{\G{c-a}\G{c-b}}.$$
For more information about the finite field analogue of
(\ref{eq1.1}), please consult, for example, \cite{FLR, DM}.

There are many double hypergeometric functions which are important
in the field of hypergeometric functions \cite{EH,SK}. Appell's four
functions \cite{AP, AP1, PJ}, which are shown as follows, may be the
most famous and well-known functions:
$$F_1(a;b,b';c;x,y)=\sum_{m\geq0}\sum_{n\geq0}\f{(a)_{m+n}(b)_m(b')_n}{m!n!(c)_{m+n}}x^my^n, \qq |x|,|y|<1,$$
$$F_2(a;b,b';c,c';x,y)=\sum_{m\geq0}\sum_{n\geq0}\f{(a)_{m+n}(b)_m(b')_n}{m!n!(c)_{m}(c')_n}x^my^n, \qq |x|+|y|<1,$$
$$F_3(a,a';b,b';c;x,y)=\sum_{m\geq0}\sum_{n\geq0}\f{(a)_m(a')_n(b)_m(b')_n}{m!n!(c)_{m+n}}x^my^n, \qq
|x|,|y|<1,$$
$$F_4(a;b;c,c';x,y)=\sum_{m\geq0}\sum_{n\geq0}\f{(a)_{m+n}(b)_{m+n}}{m!n!(c)_{m}(c')_n}x^my^n,
\qq |x|^{\f12}+|y|^{\f12}<1.$$

For more details about Appell series, please refer to \cite{JP, MJ,
W1}. The $F_1$ function has an integral representation in terms of a
single integral \cite[9.3(4)]{B2}:
$$F_1(a;b,b';c;x,y)=\f{\G c}{\G a\G{c-a}}\int_0^1u^{a-1}(1-u)^{c-a-1}(1-ux)^{-b}(1-uy)^{-b'}du,$$ where $0< Re(a)< Re(c).$
Motivated by Greene \cite{GJ} we attempt to give a finite field
analogue of $F_1(a;b,b';c;x,y)$. Although the other three series
have integral representations with a single integral (involving
products of $_2F_1$'s) and simple double integral representations,
we cannot give simple and beautiful analogues of them easily.

\begin{df}\label{df1.1} For any multiplicative characters $A,B,B',C\in\Fq$ and any $x,y\in\Bbb
F_q,$ we define
$$\Fa_1(A;B,B';C;x,y)=\varepsilon(xy)AC(-1)\sum\limits_uA(u)\overline{A}C(1-u)\overline{B}(1-ux)\overline{B'}(1-uy).$$
\end{df}
In this definition we have dropped the constant $\G c/\G a\G{c-a}$
in order to obtain simpler results. The factor
$\varepsilon(xy)AC(-1)$ is chosen so as to lead to a better
expression in terms of binomial coefficients. In the following
theorem we give an another representation for $\Bbb
F_1(A;B,B';C;x,y)$ involving binomial coefficients.
\begin{thm}\label{thm1.3} For characters $A,B,B',C \in \Fq$ and $x,y\in\Bbb F_q$ we have
$$\Fa_1(A;B,B';C;x,y)=\f1{(q-1)^2}\sum_{\chi,\lambda}\b{A\chi\lambda}{C\chi\lambda}\b{B\chi}\chi\b{B'\lambda}
\lambda\chi(x)\lambda(y).$$
\end{thm}
From Definition \ref{df1.1} we get the following corollary immediately.
\begin{cor}\label{cor1.1} For characters $A,B,B',C \in \Fq$ and $x,y\in\Bbb F_q$ we
obtain
\begin{equation}\label{eq1.3}\Fa_1(A;B,B';C;x,y)=\Fa_1(A;B',B;C;y,x),\end{equation}
\begin{equation}\label{eq3.3}\Bbb F_1(A;B,B';C;x,x)=\ _2\Bbb F_1\left[\begin{array}{rr}BB',&A
\\&C\end{array};x\right]\end{equation}
 and
\begin{equation}\label{eq1.5}\Bbb F_1(A;B,B';C;x,1)=B'(-1)\ _2\Bbb
F_1\left[\begin{array}{rr}B,&A\\&\overline{B'}C\end{array};x\right].\end{equation}
\end{cor}
%Combining the definition of Appell series and  Theorem \ref{thm1.1} we attempt to define
%$$\Bbb F_2(A;B,B';C,C';x,y)=\sum_{\chi,\lambda}\b{A\chi}\chi\b{A\chi\lambda}\lambda\b{B\chi}{C\chi}\b{B'\lambda}{C'\lambda}\chi(x)\lambda(y),
%$$
%$$\Bbb F_3(A,A';B,B';C;x,y)=\sum_{\chi,\lambda}\b{A\chi}\chi\b{B\chi}{C\chi}\b{A'\lambda}\lambda\b{B'\lambda}{C\chi\lambda}\chi(x)
%\lambda(y)$$ and
%$$\Bbb F_4(A;B;C,C';x,y)=\sum_{\chi,\lambda}\b{A\chi}\chi\b{B\chi}{C\chi}\b{A\chi\lambda}\lambda
%\b{B\chi\lambda}{C'\lambda}\chi(x)\lambda(y).$$

The main goal of this paper is to give some transformation and
reduction formulas and generation functions for
$\Fa_1(A;B,B';C;x,y).$ Several examples of such formulas are
\begin{equation*}\Fa_1(A;B,B';C;x,y)=C(-1)\overline B(1-x)\overline{B'}(1-y)\Fa_1
\left(\overline AC;B,B';C;\f x{x-1},\f y{y-1}\right),\end{equation*}
\begin{align*}\varepsilon(x-y)\Bbb F_1(A;B,B';BB';x,y)&=\varepsilon(xy)\overline A(1-x)\ _2\Bbb
F_1\left[\begin{array}{rr}B',&A\\&BB'\end{array};\f{y-x}{1-x}\right]\\&\q
-\varepsilon(y-x)\overline B(-x)\overline {B'}(-y)\end{align*} and
\begin{equation*}
\begin{split}&\f1{q-1}\sum_{\theta}\b{A\overline C\theta}\theta\Bbb F_1(A\theta;B,B';C;x,y)\theta(t)\\&=\varepsilon(t)
\overline A(1-t)\Bbb F_1\left(A;B,B';C;\f x{1-t},\f
y{1-t}\right)-\varepsilon(xy)\overline AC(-t) \overline
B(1-x)\overline{B'}(1-y)
\end{split}\end{equation*}
as given in Theorem 3.2, Corollary 3.2 and Theorem 4.1,
respectively.

 The rest of this paper is organized  as follows. Section 2 is
devoted to the proof of Theorem \ref{thm1.3}. In Section 3 we give
several transformation and reduction formulas for $\Bbb
F_1(A;B,B';C;x,y).$ Section 4 is devoted to the generating functions
for  $\Bbb F_1(A;B,B';C;x,y)$.

\section{Proof of Theorem \ref{thm1.3}}

\q First we list the following three propositions which will be used frequently without indiction in this paper.
\begin{pro}\label{pro2.1}{\rm{\cite[(2.6) and (2.15)]{GJ}}} If $A,B,C\in\Fq$ then
$$\b AB=\b A{A\overline B},$$
$$\b CA\b AB=\b CB\b{C\overline B}{A\overline B}-(q-1)(B(-1)\delta(A)-AB(-1)\delta(B\overline C)).$$
\end{pro}

\begin{pro}\label{pro2.2}{\rm{\cite[(2.10)]{GJ}} (\textbf{Binomial
Theorem})} For character $A$ of $\widehat{\Bbb F_q}$ and $x\in \Bbb
F_q$
$$\overline A(1-x)=\delta(x)+\f1{q-1}\sum_{\chi}\b{A\chi}\chi\chi(x),$$
 where the sum ranges over all multiplicative characters of $\widehat{\Bbb F_q}$ and $\delta$ is the function
 $$\delta(x)=\begin{cases}1\q&\text{if}\q x=0,\\0&\text{if}\q x\neq0.\end{cases}$$
\end{pro}
\begin{pro}\label{pro2.3}{\rm{\cite[p. 89-90]{KM}}(\textbf{Orthogonal Relations})} If $\chi\in\Fq$ and $t\in\Bbb F_q$  then
$$\sum_\chi\chi(t)=(q-1)\delta(t-1)$$and$$\sum_t\chi(t)=(q-1)\delta(\chi),$$
where, for characters, we define $\delta(\varepsilon)=1$ and $\delta(\chi)=0$ for all $\chi\neq
\varepsilon.$
\end{pro}

\begin{proof}[Proof of Theorem \ref{thm1.3}:] From the binomial theorem
over finite fields we know that
$$\overline{B}(1-ux)=\delta(ux)+\f1{q-1}\sum_{\chi}\b{B\chi}\chi\chi(ux).$$
Since $\varepsilon(x)A(u)\delta(ux)=0$ and
$\varepsilon(y)A(u)\delta(uy)=0$ for all $x,y\in\Bbb F_q$ then
\begin{align*}&\Fa_1(A;B,B';C;x,y)\\&=\varepsilon(xy)AC(-1)\sum_{u}A(u)\overline{A}C(1-u)\big(\delta(ux)+
\f1{q-1}\sum_{\chi}\b{B\chi}\chi\chi(ux)\big)\\&\q\times\big(\delta(uy)+\f1{q-1}
\sum_{\lambda}\b{B'\lambda}\lambda\lambda(uy)\big)\\&=\f{AC(-1)}{(q-1)^2}\sum_{u,\chi,\lambda}A\chi\lambda(u)\overline{A}C(1-u)\b{B\chi}
\chi\b{B'\lambda}\lambda\chi(x)\lambda(y)\\&
=\f{AC(-1)}{(q-1)^2}\sum_{\chi,\lambda}J(A\chi\lambda,\overline{A}C)\b{B\chi}\chi\b{B'\lambda}\lambda\chi(x)\lambda(y)\\&=
\f1{(q-1)^2}\sum_{\chi,\lambda} \b{A\chi\lambda}{A\overline
C}\b{B\chi}\chi\b{B'\lambda}\lambda\chi(x)\lambda(y)\\&=
\f1{(q-1)^2}\sum_{\chi,\lambda}\b{A\chi\lambda}{C\chi\lambda}\b{B\chi}\chi\b{B'\lambda}
\lambda\chi(x)\lambda(y).
\end{align*} This completes the proof.\end{proof}

 \section{Transformation and reduction formulas for $\Fa_1$}
 \q As we know, only some special hypergeometric series have summation formulas. So studying their transformation and
 reduction formulas is becoming more important and it is also true over finite fields. In this section we obtain
 some transformation and reduction formulas for $\Bbb F_1(A;B,B';C;x,y).$

 Just by some simple calculation we will obtain several reduction formulas immediately, which can be considered as
 the finite field analogues of
$$F_1(a;b,0;c;x,y)=\
_2F_1\left[\begin{array}{rr}a,&b\\&c\end{array};x\right]$$ and
$$F_1(a;0,b';c;x,y)=\ _2F_1\left[\begin{array}{rr}a,&b'\\&c\end{array};y\right],$$
respectively.

 \begin{thm}\label{thm3.1}For $A,B,B',C\in\Fq$ and $x,y\in\Bbb F_q$ we have
\begin{equation}\label{eq3.1}\Bbb F_1(A;B,\varepsilon;C;x,y)=\varepsilon(y)\ _2\Bbb F_1
\left[\begin{array}{rr}B,&A\\&C\end{array};x\right]-\varepsilon(x)\overline
AC(1-y)
 B\overline C(y)\overline B(y-x),
\end{equation}
\begin{equation}\label{eq3.2}
\Bbb F_1(A;\varepsilon,B';C;x,y)=\varepsilon(x)\ _2\Bbb F_1\left[\begin{array}{rr}B',&A\\&C
\end{array};y\right]-\varepsilon(y)\overline AC(1-x)B'\overline C(x)\overline {B'}
(x-y).\end{equation}

\end{thm}
\begin{proof} First we  prove (\ref{eq3.1}).
Obviously, (\ref{eq3.1}) holds when $y=0$. For $y\neq0$ we have
\begin{align*}
&\Bbb
F_1(A;B,\varepsilon;C;x,y)=\varepsilon(x)AC(-1)\sum_uA(u)\overline
AC(1-u)\overline B(1-ux)\varepsilon(1-uy)\\&
=\varepsilon(x)AC(-1)\sum_{u\neq\f1y}A(u)\overline AC(1-u)\overline
B(1-ux)\\&= \varepsilon(x)AC(-1)\Big(\sum_uA(u)\overline
AC(1-u)\overline B(1-ux)-A(1/y)\overline AC(1-1/y)\overline
B(1-x/y)\Big)\\& =\ _2\Bbb
F_1\left[\begin{array}{rr}B,&A\\&C\end{array};x\right]-\varepsilon(x)\overline
AC(1-y) B\overline C(y)\overline B(y-x).
\end{align*}
Combing (\ref{eq1.3}) with (\ref{eq3.1}) we can get (\ref{eq3.2})
immediately. This completes the proof.
\end{proof}

In \cite{MJ} Schlosser stated the following three transformation
formulas for $F_1$.
\begin{equation}\label{eq3.4}F_1(a;b,b';c;x,y)=(1-x)^{-b}(1-y)^{-b'}F_1(c-a;b,b';c;\f x{x-1},\f y{y-1}),\end{equation}
 \begin{equation}\label{eq3.5}F_1(a;b,b';c;x,y)=(1-x)^{-a}F_1\left(a;-b-b'+c,b';c;\f x{x-1},\f{y-x}{1-x}\right)\end{equation}
and
\begin{equation}\label{eq3.6}F_1(a;b,b';c;x,y)=(1-x)^{c-a-b}(1-y)^{-b'}F_1\left(c-a;c-b-b',b';c;x,\f{x-y}{1-y}\right).
\end{equation}

In the following theorem we give finite field analogues of
(\ref{eq3.4})-(\ref{eq3.6}).
\begin{thm}\label{thm3.3} For characters $A,B,B'$ and $C$ of $\Fq$
and $x,y\in\Bbb F_q$ we have
\begin{equation}\label{eq3.4'}\Fa_1(A;B,B';C;x,y)=C(-1)\overline B(1-x)\overline{B'}(1-y)\Fa_1\left(\overline AC;B,B';C;\f x{x-1},\f y{y-1}\right),\end{equation}
\begin{equation}\label{eq3.5'}\varepsilon(x-y)\Bbb F_1(A;B,B';C;x,y)=\varepsilon(y)\overline A(1-x)\Bbb F_1\left(A;\overline B\overline {B'}C,B';C;\f x{x-1},\f{y-x}{1-x}\right),\end{equation}
\begin{equation}\label{eq3.6'}\begin{split}\varepsilon(x-y)\Bbb F_1(A;B,B';C;x,y)&=\varepsilon(y)C(-1)\overline A\overline BC(1-x)\overline{B'}(1-y)
\\&\q\times\Bbb F_1\left(\overline AC;\overline B\overline{B'}C,B';C;x,\f{x-y}{1-y}\right).
\end{split}\end{equation}
\end{thm}
\begin{proof}First we prove (\ref{eq3.4'}). We have
\begin{align*}&C(-1)\overline B(1-x)\overline{B'}(1-y)\Fa_1\left(\overline AC;B,B';C;\f x{x-1},\f y{y-1}\right)\\&
=\varepsilon\left(\f {xy}{(x-1)(y-1)}\right)AC(-1)\overline
B(1-x)\overline{B'}(1-y)\\&\q\times\sum\limits_{u}\overline{A}C(u)A(1-u)\overline{B}\left(1-\f{ux}{x-1}\right)\overline{B'}
\left(1-\f{uy}{y-1}\right)\\&=\varepsilon(
xy)AC(-1)\sum\limits_{u}\overline{A}C(u)A(1-u)\overline{B}\big((u-1)x+1\big)
\overline{B'}\big((u-1)y+1\big)\\&=\varepsilon(
xy)AC(-1)\sum\limits_{v}A(v)\overline{A}C(1-v)\overline{B}(1-vx)\overline{B'}(1-vy)\\&=
\Bbb F_1(A;B,B';C;x,y).
\end{align*}

Now we give the proof of (\ref{eq3.5'}). Obviously, it holds when
$x=0$. For $x\neq0,$ substituting $u=\f {v(1-x)}{1-vx}$ into the
right hand of (\ref{eq3.5'}) we have
\begin{align*}&\varepsilon(y)\overline A(1-x)
\Bbb F_1\left(A;\overline{BB'}C,B';C;\f
x{x-1},\f{y-x}{1-x}\right)\\&
=\varepsilon(xy)\varepsilon(y-x)\overline
A(1-x)AC(-1)\\&\q\times\sum_{u} A(u)\overline AC(1-u)BB'\overline
C\left(1-\f {ux}{x-1}\right)\overline{B'}
\left(1-\f{u(y-x)}{1-x}\right)
\\&=\varepsilon(xy)\varepsilon(y-x)AC(-1)\sum_{v}A(v)\overline AC(1-v)\overline
B(1-vx)\overline {B'}(1-vy)\\&=\varepsilon(y-x)\Bbb
F_1(A;B,B';C;x,y).
\end{align*}
%\begin{align*}&\varepsilon(y)C(-1)\overline A\overline BC(1-x)\overline{B'}(1-y)\Bbb F_1\left(\overline AC;\overline B\overline{B'}C;B';C;x,\f{x-y}{1-y}\right)
%\\&=\varepsilon(y)\varepsilon\left(\f{x^2-xy}{1-y}\right)\overline A\overline BC(1-x)\overline{B'}(1-y)\f{AC(-1)}{q-1}\\&\q\times\sum_{u}\overline AC(u)A(1-u)
%BB'\overline C(1-ux)\overline{B'}\left(1-u\f{x-y}{1-y}\right)
%\\&=\varepsilon(x-y)\varepsilon(y)\f{AC(-1)}{q-1}\sum_{v}A(v)\overline AC(1-v)\overline B(1-vx)\overline {B'}(1-vy)\\&=\varepsilon(x-y)\Bbb F_1(A;B,B';C;x,y).
%\end{align*}
The proof of (3.8) is similar to (3.7), but with the substitution
$u=\f{1-v}{1-vx}$ on the right hand side. This completes the proof.
\end{proof}
Combining Theorem \ref{thm3.3} with \eqref{eq1.3} we obtain the
following theorem.
\begin{thm}\label{thm3.4} For characters $A,B,B',C$ of $\Fq$ and $x,y\in \Bbb
F_q$ then
\begin{equation}\label{eq3.11}\varepsilon(x-y)\Bbb F_1(A;B,B';C;x,y)=\varepsilon(x)\overline A(1-y)\Bbb F_1\left(A;B,\overline {BB'}C;C;\f{x-y}{1-y},\f y{y-1}\right)\end{equation} and
\begin{equation}\label{eq3.12}\begin{split}\varepsilon(x-y)\Bbb F_1(A;B,B';C;x,y)&=\varepsilon(x)C(-1)\overline B(1-x)\overline{AB'}C(1-y)
\\&\q\times\Bbb F_1\left(\overline AC;B,\overline {BB'}C;C;\f{y-x}{1-x},y\right)
\end{split}\end{equation} hold.
\end{thm}
\begin{proof} The proofs of  (\ref{eq3.11}) and (\ref{eq3.12}) are similar so the proof of (\ref{eq3.12}) is skipped.
 We have
\begin{align*}&\varepsilon(x-y)\Bbb F_1(A;B,B';C;x,y)=\varepsilon(x-y)\Bbb F_1(A;B',B;C;y,x)\\&=\varepsilon(x)\overline A(1-y)\Bbb F_1\left(A;\overline{BB'}C,B;C;\f y{y-1},\f{x-y}{1-y}\right)\\&=\varepsilon(x)\overline A(1-y)\Bbb F_1\left(A;B,\overline{BB'}C;C;\f{x-y}{1-y},\f y{y-1}\right).
\end{align*}Then the proof is completed.
\end{proof}
\textbf{Remark 1.} Theorem \ref{thm3.4} gives finite field analogues
of the following two transformation formulas \cite[(27, 29)]{MJ}:
$$F_1(a;b,b';c;x,y)=(1-y)^{-a}F_1\left(a;b,c-b-b';c;\f{x-y}{1-y},\f y{y-1}\right)$$  and
$$F_1(a;b,b';c;x,y)=(1-x)^{-b}(1-y)^{c-a-b'}F_1\left(c-a;b,c-b-b';c;\f{y-x}{1-x},y\right).$$

Putting $B'=\varepsilon$ into (\ref{eq3.4'}) and using Theorem
\ref{thm3.1} we obtain the following analogue of the well-known
Pfaff-Kummer transformation of $_2F_1$:
\begin{equation}\label{eq3.10}_2F_1\left[\begin{array}{rr}a,&b\\&c\end{array};x\right]=(1-x)^{-b}\ _2F_1\left[\begin{array}{rr}c-a,&b\\&c\end{array};\f x{x-1}\right].\end{equation}
\begin{cor}\label{cor3.1}For $x\in\Bbb F_q\setminus\{1\}$ we have
$$\ _2\Bbb F_1\left[\begin{array}{rr}B,&A\\&C\end{array};x\right]=C(-1)\overline
B(1-x)\ _2\Bbb F_1\left[\begin{array}{rr}B,&\overline
AC\\&C\end{array};\f x{x-1}\right].$$
 %and \begin{align*}\varepsilon(y)\ _2\Bbb F_1\left[\begin{array}{rr}B,&A\\&C\end{array};x\right]&=\varepsilon(y)\overline A(1-x)\ _2\Bbb F_1\left[\begin{array}{rr}\overline BC,&B\\&C\end{array};\f x{x-1}\right]\\\q&+\f{\varepsilon(x)}{q-1}\overline AC(y-1)B\overline C(y)\overline B(y-x),\q(x\neq y).\end{align*}
\end{cor}
\textbf{Remark 2.} In \cite{GJ} Greene also obtained the analogue of
(\ref{eq3.10}) which is read as
$$_2\Bbb F_1\left[\begin{array}{rr}A,&B\\&C\end{array};x\right]=C(-1)\overline A(1-x)\ _2\Bbb F_1\left[\begin{array}{rr}
A,&C\overline B\\&C\end{array};\f x{x-1}\right]+A(-1)\b B{\overline AC}\delta(1-x)$$
 %and $$_2\Bbb F_1\left[\begin{array}{rr}A,&B\\&C\end{array};x\right]=\overline B(1-x)\ _2\Bbb F_1\left[\begin{array}{rr}C\overline A,&B\\&C\end{array};\f x{x-1}\right]+A(-1)\b B{\overline AC}\delta(1-x).$$
and the above equation extends the domain of Corollary \ref{cor3.1}
to $\Bbb F_q.$
 \begin{cor}\label{cor3.2} For characters $A,B,B'$ of $\Fq$ and $x,y\in F_q$ we have
\begin{align*}\varepsilon(x-y)\Bbb F_1(A;B,B';BB';x,y)&=\varepsilon(xy)\overline A(1-x)\ _2\Bbb
F_1\left[\begin{array}{rr}B',&A\\&BB'\end{array};\f{y-x}{1-x}\right]\\&\q
-\varepsilon(y-x)\overline B(-x)\overline {B'}(-y)\end{align*} and
\begin{align*}
\varepsilon(x-y)&\Bbb
F_1(A;B,B';BB';x,y)=\varepsilon(xy)BB'(-1)\overline AB'(1-x)
\overline{B'}(1-y)\\&\q\times\ _2\Bbb
F_1\left[\begin{array}{rr}B',&\overline ABB'\\&BB'\end{array};
\f{x-y}{1-y}\right]-\varepsilon(x-y)\overline
B(-x)\overline{B'}(-y).\end{align*}
\end{cor}
\begin{proof} Putting $C=BB'$ into (\ref{eq3.5'}) and (\ref{eq3.6'}), respectively, and applying Theorem \ref{thm3.1}, we can obtain the
desired result.
 \end{proof}
 \textbf{Remark 3.} Corollary \ref{cor3.2} gives
finite field analogues of the following reduction formulas
  (see, \cite[26a]{MJ}):
$$ F_1(a;b,b';b+b';x,y)=(1-x)^{-a}\
_2F_1\left[\begin{array}{rr}a,&b'\\&b+b'\end{array};\f{y-x}{1-x}\right]$$
and $$F_1(a;b,b';b+b';x,y) =(1-x)^{b'-a}(1-y)^{-b'}\
_2F_1\left[\begin{array}{rr}b+b'-a,&b'\\&b+b'\end{array};\f{x-y}{1-y}\right].$$

\begin{thm}\label{thm3.7} For characters $A,B,B',C$ of $\Fq$ and $x,y\in \Bbb F_q\setminus\{0,1\}$ then
$$\Bbb F_1(A;B,B';C;x,y)=BB'(-1)\Bbb F_1(A;B,B';ABB'\overline C;1-x,1-y).$$
\end{thm}
\begin{proof} Setting $u=\f v{v-1}$ in the definition of $\Bbb F_1(A;B,B';C;x,y)$ and applying
$\varepsilon(x)=\varepsilon(1-x)$ we have
\begin{align*}&\Bbb F_1(A;B,B';C;x,y)=\varepsilon(xy)AC(-1)\sum_uA(u)\overline AC(1-u)\overline B(1-ux)\overline{B'}(1-uy)
\\&=\varepsilon((1-x)(1-y))C(-1)\sum_vA(v)BB'\overline C(1-v)\overline
B\big(1-v(1-x)\Big)\overline{B'}\Big(1-v(1-y)\Big)\\&= BB'(-1)\Bbb
F_1(A;B,B';ABB'\overline C;1-x,1-y).
\end{align*}
The proof is completed.
\end{proof}

If we set $y=1$ and $B'=\varepsilon$ in Theorem \ref{thm3.7} and combine this with
(\ref{eq1.5}) and the identity $\varepsilon(x)=\varepsilon(1-x)+\delta(1-x)-\delta(x)$ we obtain
Theorem 4.4 from \cite{GJ}:
\begin{cor}\label{cor3.3} For characters $A,B,C$ of $\Fq$ and $x\in \Bbb F_q$
we have
\begin{align*}_2\Bbb F_1\left[\begin{array}{rr}B,&A\\&C\end{array};x\right]&=B(-1)\ _2\Bbb F_1\left[\begin{array}{rr}B,&A\\&AB\overline C\end{array};1-x\right]\\&\q
+B(-1)\b A{\overline BC}\delta(1-x)-\b AC\delta(x).
\end{align*}
\end{cor}
\section{Generating Functions for $\Fa_1$}
\q Generating functions play an important role in many fields of
Mathematics. In this section we obtain two generating functions for
$\Bbb F_1(A;B,B';C;x,y).$
\begin{thm}\label{thm4.1}For $A,B,B',C\in\Fq $ and $x,y,t\in \Bbb F_q$ we have
\begin{equation*}
\begin{split}&\f1{q-1}\sum_{\theta}\b{A\overline C\theta}\theta\Bbb F_1(A\theta;B,B';C;x,y)\theta(t)\\&=
\varepsilon(t)\overline A(1-t)\Bbb F_1\left(A;B,B';C;\f x{1-t},\f
y{1-t}\right)-\varepsilon(xy)\overline AC(-t)\overline
B(1-x)\overline{B'}(1-y).
\end{split}\end{equation*}
\end{thm}
\begin{proof}Replacing $u$ by $v(1-t)$ and using $$\overline AC\left(1+\f{vt}{1-v}\right)=\delta\left(\f{vt}{1-v}\right)+
\f1{q-1}\sum_{\theta}\b{A\overline
C\theta}\theta\theta\left(\f{-vt}{1-v}\right)$$ we obtain
\begin{align*}&\varepsilon(t)\overline A(1-t)\Bbb F_1\left(A;B,B';C;\f x{1-t},\f y{1-t}\right)
\\&=\varepsilon(t)\varepsilon\left(\f{xy}{(1-t)^2}\right)AC(-1)\sum_{u}A(u)\overline
AC(1-u)\overline B\left(1-\f {ux}{1-t}\right)\overline
{B'}\left(1-\f {uy}{1-t}\right) \overline
A(1-t)\\&=\varepsilon(t)\varepsilon(xy)AC(-1)\sum_{v}A(v)\overline
 AC(1-v+vt)\overline B(1-vx)\overline{B'}(1-vy)\\&=\varepsilon(t)\varepsilon(xy)AC(-1)
 \Big(\sum_{v\neq1}A(v)\overline AC(1-v)\overline B(1-vx)\overline{B'}(1-vy)
 \overline AC\left(1+\f{vt}{1-v}\right)\\&\q+\overline AC(t)\overline B(1-x)\overline{B'}(1-y)
 \Big)\\&=\f1{q-1}\sum_{\theta}\b{A\overline C\theta}\theta\varepsilon(xy)AC
\theta(-1)\sum_{v}A\theta(v)\overline A\overline \theta
C(1-v)\overline B(1-vx)
\overline{B'}(1-vy)\theta(t)\\&\q+\varepsilon(xy)\overline
AC(-t)\overline
B(1-x)\overline{B'}(1-y)\\&=\f1{q-1}\sum_{\theta}\b{A\overline
C\theta} \theta\Bbb
F_1(A\theta;B,B';C;x,y)\theta(t)+\varepsilon(xy)\overline
AC(-t)\overline B(1-x)\overline{B'}(1-y).
\end{align*}
This completes the proof.
\end{proof}
\textbf{Remark 4.} Theorem \ref{thm4.1} could be regarded as the
finite field analogue of \cite{JP}
$$\sum_{k=0}^\infty\b{a+k-1}kF_1(a+k;b,b';c;x,y)t^k=(1-t)^{-a}F_1\left(a;b,b';c;\f
x{1-t},\f y{1-t}\right),\q (|t|<1).$$
%Combining Theorem 3.1 with (\ref{eq3.4'}) we obtain the following transformation formula for $\Bbb F_1(A;B,B';C;x,y)$ by equating the coefficients of $\theta(t)$ on the both sides of equation in the Theorem 4.1.
%\begin{thm}If $\theta\neq\overline AC$ then
%$$\b{A\overline C\theta}\theta\Bbb F_1(A\theta;B,B';C;x,y)=C(-1)B(1-x)B'(1-y)\b{A\theta}\theta\Bbb F_1(A;B,B';C;x,y).$$
%\end{thm}
%\begin{thm}$$\b{A\overline C\theta}\theta\b{A\theta\chi\lambda}{C\chi\lambda}\b{B\chi}\chi\b{B'\lambda}\lambda=C(-1)\b{A\theta}\theta\sum_{\eta,\kappa}
%\b{B\eta}\eta\b{B'\kappa}\kappa\b{A\overline{\eta\kappa}\chi\lambda}
%{A\overline C}\b{B\overline \eta\chi}B\b{B'\overline\kappa\lambda}{B'}.$$
%\end{thm}
%\begin{proof}If $xy\neq0$ then
%\begin{align*}&\b{A\overline C\theta}\theta\sum_{\chi,\lambda}\b{A\theta\chi\lambda}{C\chi\lambda}\b{B\chi}\chi\b{B'\lambda}\lambda\chi(x)\lambda(y)
%\\&=C(-1)\b{A\theta}\theta\left(\delta(x)+\sum_{\eta}\b{B\eta}\eta\eta(x)\right)\left(\delta(y)+\sum_{\kappa}\b{B'\kappa}
%\kappa\kappa(x)\right)\\&\q\times
%\sum_{\chi,\lambda}\b{A\chi\lambda}{C\chi\lambda}\b{B\chi}\chi\b{B'\lambda}\lambda\chi(x)\lambda(y)
%\\&=C(-1)\b{A\theta}\theta
%\sum_{\eta,\kappa}\b{B\eta}
%\eta\b{B'\kappa}\kappa\sum_{\chi,\lambda}\b{A\chi\lambda}{C\chi\lambda}\b{B\chi}\chi\b{B'\lambda}\lambda\eta\chi(x)
%\kappa\lambda(y)
%\\&=C(-1)\b{A\theta}\theta\sum_{\chi,\lambda}\sum_{\eta,\kappa}\b{B\eta}\eta\b{B'\kappa}\kappa\b{A\overline{\eta\kappa}\chi\lambda}
%{A\overline C}\b{B\overline \eta\chi}B\b{B'\overline\kappa\lambda}{B'}\chi(x)\lambda(y).
%\end{align*}
%Then we complete the proof.
%\end{proof}
\begin{thm}\label{thm4.2} For $A,B,B',C\in\Fq $ and $x,y,t\in \Bbb F_q$ we obtain
\begin{align*}\f1{q-1}\sum_\theta\b{B\theta}\theta\Bbb F_1(A;B\theta,B';C;x,y)\theta(t)&=
\varepsilon(t)\overline B(1-t)\Bbb F_1\left(A;B,B';C;\f
x{1-t},y\right)\\&\q-\varepsilon(y)\overline B(-t) B'\overline
C(x)\overline AC(1-x)\overline{B'}(x-y).
\end{align*}
\end{thm}
\begin{proof}
 Obviously, the above equation holds when $x=0$ or $y=0$. So we only need to consider the case
 $xy\neq0.$
Since
\begin{align*}&\b{B\theta}\theta\b{B\chi\theta}\chi=\b{B\theta}\theta\b{B\chi\theta}{B\theta}\\
&=\b{B\chi\theta}\theta\b{B\chi}\chi-(q-1)(\theta(-1)\delta(B\theta)-B(-1)\delta(B\chi)),\end{align*}
then combining the binomial theorem over finite fields we obtain
\begin{align*}&\f1{q-1}\sum_\theta\b{B\theta}\theta\Bbb F_1(A;B\theta,B';C;x,y)\theta(t)-\varepsilon(t)\overline B(1-t)\Bbb F_1\left(A;B,B';C;\f x{1-t},y\right)
\\&=\f1{(q-1)^3}\sum_{\chi,\lambda,\theta}\left(\b{B\theta}\theta\b{B\chi\theta}\chi-\b{B\chi\theta}\theta\b{B\chi}\chi\right)
\b{A\chi\lambda}
{C\chi\lambda}\b{B'\lambda}\lambda\chi(x)\lambda(y)\theta(t)\\&=\f1{(q-1)^2}\sum_{\chi,\lambda,\theta}\big(B(-1)
\delta(B\chi)-\theta(-1)\delta(B\theta)\big)\b{A\chi\lambda}
{C\chi\lambda}\b{B'\lambda}\lambda\chi(x)\lambda(y)\theta(t)\\&=-\f{\overline
B(-t)}{(q-1)^2}
\sum_{\chi,\lambda}\b{A\chi\lambda}{C\chi\lambda}\b{B'\lambda}\lambda\chi(x)\lambda(y)\\&
=-\f{\overline
B(-t)}{(q-1)^2}\sum_\lambda\b{B'\lambda}\lambda\lambda(y)\overline{C\lambda}(x)
\sum_\chi\b{A\overline C\chi}\chi\chi(x)\\&= -\f{\overline
B(-t)\overline
C(x)}{q-1}\sum_\lambda\b{B'\lambda}\lambda\lambda(y/x)\big(\overline
AC(1-x)-\delta(x)\big)
\\&=-\overline B(-t)\overline C(x)\overline
AC(1-x)\left(\overline{B'} (1-y/x)-\delta(y/x)\right)
\\&=-\overline B(-t)B'\overline C(x)\overline AC(1-x)\overline{B'}(x-y).
\end{align*}
This completes the proof of this theorem.
\end{proof}

From Theorems \ref{thm4.1} and \ref{thm4.2} we can derive two
generating functions for $_2\Bbb F_1$ which are shown in the
following theorem.
\begin{thm} For $A,B,C\in\Fq $ and $x,y,t\in \Bbb F_q$ we have
\label{thm4.4}
\begin{equation}\label{eq4.1}
\begin{split}\f1{q-1}\sum_\theta\b{A\overline C\theta}\theta\ _2\Bbb F_1\left[\begin{array}{rr}B,
&A\theta\\&C\end{array};x\right]\theta(t)&=\varepsilon(t)\overline
A(1-t)\ _2\Bbb F_1\left[\begin{array}{rr}B,&A\\&C\end{array};\f
x{1-t}\right]\\&\q-\varepsilon(x)A(1-t)\overline AC(t)\overline
B(1-x)\end{split}
\end{equation}
and
\begin{equation}\label{eq4.2}\begin{split}\f1{q-1}\sum_\theta\b{B\theta}\theta\ _2\Bbb F_1\left[\begin{array}{rr}B\theta,&A\\&C
\end{array};x\right]\theta(t)&=\varepsilon(t)\overline B(1-t)\ _2\Bbb F_1\left[\begin{array}{rr}B,&A\\&C\end{array};\f x{1-t}\right]
\\&\q-\overline B(-t)\overline A C(1-x)\overline C(x).
\end{split}
\end{equation}
\end{thm}
\begin{proof} First we prove (\ref{eq4.1}). Setting $y=1$ and $B'=\varepsilon$ into Theorem \ref{thm4.1} and applying
(\ref{eq1.5}) we have
\begin{align*}&\f1{q-1}\sum_\theta\b{A\overline C\theta}\theta\ _2\Bbb F_1\left[\begin{array}{rr}B,&A\theta\\&C\end{array};x\right]\theta(t)=\varepsilon(t)\overline A(1-t)
\Bbb F_1\left(A;B,\varepsilon;C;\f
x{1-t},\f1{1-t}\right)\\&=\varepsilon(xt)\overline
A(1-t)AC(-1)\sum_uA(u)\overline AC(1-u) \overline
B\left(1-\f{ux}{1-t}\right)\varepsilon\left(1-\f u{1-t}\right)
\\&=\varepsilon(xt)\overline A(1-t)AC(-1)\sum_{u\neq1-t}A(u)\overline AC(1-u)\overline B\left(1-\f{ux}{1-t}\right)
\\&=\varepsilon(xt)\overline A(1-t)AC(-1)\sum_uA(u)\overline AC(1-u)\overline B\left(1-\f{ux}{1-t}\right)-\varepsilon(x)
\overline AC(-t)\overline B(1-x)\\&=\varepsilon(t)\overline A(1-t)\
_2\Bbb F_1\left[\begin{array}{rr}B,&A\\&C\end{array};\f
x{1-t}\right]-\varepsilon(x)\overline AC(-t)\overline B(1-x).
\end{align*}
The proof of (\ref{eq4.2}) is similar, setting $y = 1$ and $B' =
\varepsilon$ in Theorem \ref{thm4.2}. The details are omitted.
\end{proof}
\end{document}